\documentclass[12pt]{amsart}
\usepackage{amssymb}
\usepackage{amscd}
\usepackage{amsmath}
\usepackage{amsfonts}
\usepackage{latexsym}
\usepackage[all]{xy}
\usepackage{verbatim}
\usepackage[margin=1in]{geometry}

\newtheorem{theorem}{Theorem}[section]
\newtheorem{lemma}[theorem]{Lemma}
\newtheorem{proposition}[theorem]{Proposition}
\newtheorem{corollary}[theorem]{Corollary} 
\theoremstyle{definition}  
\newtheorem{definition}[theorem]{Definition}
\newtheorem{example}[theorem]{Example}
\newtheorem{conjecture}[theorem]{Conjecture}  

\newtheorem{remark}[theorem]{Remark}

\newcommand{\id}{\text{id}}

\newcommand{\Fun}{\text{Fun}}
\newcommand{\FPdim}{\text{FPdim}} 

\renewcommand{\Vec}{\operatorname{\operatorname{\mathsf{Vec}}}}

\DeclareMathOperator{\Ind}{\operatorname{\mathsf{Ind}}}
\DeclareMathOperator{\Rep}{\operatorname{\mathsf{Rep}}}
\DeclareMathOperator{\Hom}{\operatorname{\mathsf{Hom}}}
\DeclareMathOperator{\Alt}{\operatorname{\mathsf{Alt}}}

\newcommand{\rev}{\text{rev}}

\newcommand{\op}{\text{op}}
\newcommand{\sgn}{\text{sgn}}

\newcommand{\eps}{\varepsilon}

\newcommand{\B}{\mathcal{B}}
\newcommand{\C}{\mathcal{C}}
\newcommand{\D}{\mathcal{D}}
\newcommand{\E}{\mathcal{E}}
\newcommand{\F}{\mathcal{F}}
\newcommand{\Z}{\mathcal{Z}}

\newcommand{\M}{\mathcal{M}}

\renewcommand{\O}{\mathcal{O}}

\newcommand{\be}{\mathbf{1}}

\renewcommand{\be}{\mathbf{1}}

\newcommand{\bt}{\boxtimes}
\newcommand{\ot}{\otimes}

\newcommand{\beq}{\begin{equation}}
\newcommand{\eeq}{\end{equation}}

\newcommand{\bpf}{\begin{proof}}
\newcommand{\epf}{\end{proof}}

\newcommand{\bth}{\begin{theorem}}
\renewcommand{\eth}{\end{theorem}}
\newcommand{\bpr}{\begin{proposition}}
\newcommand{\epr}{\end{proposition}}
\newcommand{\ble}{\begin{lemma}}
\newcommand{\ele}{\end{lemma}}
\newcommand{\bco}{\begin{corollary}}
\newcommand{\eco}{\end{corollary}}
\newcommand{\bde}{\begin{definition}}
\newcommand{\ede}{\end{definition}}
\newcommand{\bex}{\begin{example}}
\newcommand{\eex}{\end{example}}
\newcommand{\bre}{\begin{remark}}
\newcommand{\ere}{\end{remark}}
\newcommand{\bcj}{\begin{conjecture}}
\newcommand{\ecj}{\end{conjecture}}

\begin{document}
\title[Classifying braidings on  fusion categories]
{Classifying braidings on  fusion  categories}
\author{Dmitri Nikshych}
\address{Department of Mathematics and Statistics,
University of New Hampshire,  Durham, NH 03824, USA}
\email{dmitri.nikshych@unh.edu}

\begin{abstract}
We show that braidings on a fusion category $\C$ correspond to certain fusion subcategories of the center of $\C$
transversal to the canonical Lagrangian algebra. This allows to classify braidings
on  non-degenerate braided fusion categories and on those dual to the categories of group-graded vector spaces. 
\end{abstract}  

\maketitle
\baselineskip=18pt


\section{Introduction}

Throughout this article we work over an algebraically closed field $k$ of characteristic $0$.  

In general, a fusion category $\C$ may have several different braidings or no braidings at all. 
For example, if $\C=\Vec_G$, the category of finite-dimensional $k$-vectors spaces graded by a finite abelian group $G$,
then braidings on $\C$ are parameterized by bilinear forms on $G$. If $G$ is non-Abelian then of course $\Vec_G$ does not
admit any braidings. 

The goal of this note is to give a convenient parameterization of braidings on an arbitrary fusion category $\C$. 
We introduce the notion of transversality between algebras and subcategories of a braided fusion category.  Then we show
that  the set of braidings on $\C$ is in bijection with the set of fusion subcategories $\B$ of the center $\Z(\C)$ such that
$\FPdim(\B)=\FPdim(\C)$ and $\B$ is transversal to the canonical Lagrangian algebra of $\Z(\C)$. In several 
interesting situations it is possible to give an explicit parameterization of such subcategories. 
We do this in two cases: (1) for fusion categories $\C$ admitting a non-degenerate braiding and (2) for group-theoretical
categories. In the latter case the parameterization is given in terms of the subgroup lattice of a group and can be conveniently
used in concrete computations. 

The paper is organized as follows.  Section~\ref{sect 2} contains some background information and a categorical
analogue of Goursat's lemma (Theorem~\ref{Goursat}) for subcategories of tensor products of fusion categories.
In Section~\ref{sect 3} we introduce transversal pairs of algebras and subcategories and characterize braidings in 
these terms.  In Section~\ref{sect 4} we classify braidings on a fusion category $\B$ that already admits a non-degenerate braiding
(Theorem~\ref{braidings on non-degenerate}) and consider several examples.  
We show that with respect to any other braiding the symmetric center of $\B$
remains pointed. In Section~\ref{sect 5} we classify braidings on group-theoretical fusion categories (dual to the category $\Vec_G$).
As an application we parameterize  braidings on the Drinfeld center of $\Vec_G$.

{\bf Acknowledgements.} 
The author thanks Costel-Gabriel Bontea and Alexei Davydov for useful discussions. The author is grateful to the referee
for her or his useful comments.  This  material is based upon work supported  by  the  National  Science  Foundation  under  
grant DMS-1801198 and  by the NSA grant H98230-16-1-0008.


\section{Preliminaries}
\label{sect 2}

\subsection{Fusion categories}

We refer the reader to \cite{EGNO} for a general  theory of tensor categories and to \cite{DGNO} for braided fusion categories.

A {\em fusion category} over $k$ is a $k$-linear semisimple rigid tensor category with finitely many isomorphism
classes of simple objects, finite-dimensional $\Hom$-spaces, and a simple unit object $\be$.  By a {\em fusion subcategory} 
of a fusion category $\C$ we always mean a full tensor subcategory. An example of subcategory is the maximal pointed subcategory 
$\C_{pt}\subset \C$ generated by invertible objects of $\C$.  We say that $\C$ is {\em pointed} if $\C=\C_{pt}$.

We denote $\Vec$ the fusion category of finite-dimensional $k$-vector spaces.

For a fusion category $\C$ let $\O(\C)$ denote the set of isomorphism classes of simple objects. 

Let $G$ be a finite group.
  A grading of $\C$ by $G$ is a map $\deg :\O(\C)\to G$ with the following property:  for any simple objects 
  $X,\,Y,\,Z\in \C$ such that $X\ot Y$ contains $Z$ one has $\deg\,Z = \deg\, X \cdot \deg\, Y$.  We will identify a grading
  with the corresponding decomposition
  \begin{equation}
  \label{grading}
  \C =\bigoplus_{g\in G}\, \C_g,
  \end{equation}
  where $\C_g$ is the full additive subcategory of $\C$ generated by simple objects of degree $g\in G$.  The
  subcategory $\C_1$ is called the {\em trivial component} of the grading.
 The grading is called {\em faithful} if $\deg :\O(\C)\to G$ is surjective.
 
 
 For any fusion category $\C$  there is a {\em universal grading} $\O(\C)\to U(\C)$ \cite{GN}, where $U(\C)$ is the {\em universal grading group}
 of $\C$. Any grading of $\C$ comes from a  quotient of $U(\C)$. The trivial component of the universal grading
 is the adjoint fusion subcategory $\C_{ad} \subset \C$ generated by objects $X\ot X^*,\, X\in \O(\C)$.

\subsection{Fiber products  of fusion categories}
 
 Let $\C,\, \D$ be fusion categories graded by the same group $G$.  The {\em fiber product} of $\C$ and $\D$
 is the  fusion category
 \begin{equation}
 \label{fiber product}
 \C \bt_G \D := \bigoplus_{g\in G}\, \C_g \bt \D_g. 
 \end{equation}
 Here $\bt$ denotes Deligne's tensor product of abelian categories. 
Clearly, $\C \bt_G \D$ is a fusion subcategory of $\C \bt \D$ graded by $G$. 
The trivial component of this grading is $\C_1\bt \D_1$. 
When the gradings of $\C$ and $\D$ are faithful
one has
\begin{equation}
\label{FPdim fiber}
\FPdim(\C \bt_G \D)=  \frac{\FPdim(\C) \FPdim(\D)}{|G|}.
\end{equation}

\subsection{Goursat's Lemma for subcategories of the tensor product }

Let $\C,\, \D$ be fusion categories. 

\begin{definition}
\label{subcategory datum}
A {\em subcategory datum} for $\C \bt \D$ consists of a pair $\E \subset \C$ and $\F \subset \D$ 
of fusion subcategories, a group $G$, and  fixed faithful gradings of $\E$ and $\F$ by $G$.  
\end{definition}

We will identify subcategory data  $(\E,\, \F,\, G)$ and $(\E,\, \F,\, G')$ if there is an isomorphism
$\alpha:G \xrightarrow{\sim}  G'$ such that $\E_g =\E_{\alpha(g)}$ and $\F_g =\F_{\alpha(g)}$.
When no confusion
is likely we will denote a subcategory datum simply by $(\E,\, \F,\, G)$ omitting the grading maps. 

Given a  subcategory datum  $(\E,\, \F,\, G)$ we can form a fusion subcategory 
\begin{equation}
\label{S(E,F)}
\mathcal{S}(\E,\, \F,\, G) := \E \bt_G \F \subset \C \bt \D. 
\end{equation}


It turns out that $\mathcal{S}(\E,\, \F,\, G)$ is a typical example of a fusion subcategory of $\E \bt \F$. 
The following theorem is a categorical analogue of the well known Goursat's Lemma in group theory. 


\begin{theorem}
\label{Goursat}
Let $\C,\, \D$ be fusion categories. 
The assignment
\begin{equation}
\label{subcategory S}
(\E,\, \F,\, G) \mapsto \mathcal{S}(\E,\, \F,\, G) 
\end{equation}
is a bijection between the set  of subcategory data for $\C\bt\D$ and the set of fusion subcategories of $\C \bt \D$. 
\end{theorem}
\begin{proof}
We need to show that  every fusion subcategory $\mathcal{S} \subset \C \bt \D$ is equal to some $\mathcal{S}(\E,\, \F,\, G)$
for a unique choice of $(\E,\, \F,\, G)$. 

Let $\E \subset \C$  be a fusion subcategory generated by all $X\in \O(\C)$  such that
$X\bt Y \in \mathcal{S}$   for some non-zero $Y\in \D$.  Similarly, let  $\F \subset \D$
be a fusion subcategory generated by all $Y\in \O(\D)$  such that
$X\bt Y \in \mathcal{S}$   for some non-zero $X\in \C$.

Let 
\begin{equation}
\label{trivial comps}
\tilde{\E} := \mathcal{S} \cap (\C \bt \Vec)\subset \E \mbox{   and  }\tilde\F:= \mathcal{S} \cap (\Vec \bt \D)\subset \F.
\end{equation} 

If $X\in \O(\C)$ and $Y\in\O(\D)$ are such that  $X\bt Y \in  \mathcal{S}$ then $(X^*\ot X)\bt \be$
and $\be\bt (Y^*\ot Y)$ are objects of $\mathcal{S}$. This  means that $\E_{ad}\subset \tilde{\E}$ and 
$\F_{ad}\subset \tilde\F$. Let $H_\E\subset U(\E)$ and $H_\F\subset U(\F)$ be the 
subgroups of the universal groups corresponding to $\tilde\E$ and $\tilde\F$.  
We claim that these subgroups are normal. Indeed, let $X\in \O(\tilde{\E})$
and $V\in \O(\E)$. Then $X\bt \be \in \mathcal{S}$ and $V\bt U \in \mathcal{S}$ for some $U\in \O(\D)$.
So  $(V^*\bt U^*) \ot (X \bt \be) \ot (V\bt U) = (V^* \ot X \ot V) \bt (U^*\ot U) \in \mathcal{S}$
and $V^* \ot X \ot V \in \tilde{\E}$. This implies $gxg^{-1} \in H_\E$ for all $x\in H_\E$ and $g\in U(\E)$.
Thus, $H_\E \subset U(\E)$ is normal. Similarly, $H_\F \subset U(\F)$ is normal.

Hence, subcategories  $\E$ and $\F$ have faithful gradings  
$\deg_\E:\O(\E)\to U(\E)/H_\E =:G_\E$ and $\deg_\F:\O(\F)\to U(\F)/H_\F =:G_\F$
with trivial components $\tilde{\E}$ and $\tilde\F$, respectively. 

Let $X\in \O(\C)$ and $Y_1,\, Y_2 \in\O(\D)$ be such that that  $X\bt Y_1,\, X\bt Y_2 \in  \mathcal{S}$.
Then  $\be \bt (Y_1^*\ot Y_2)$ is  a subobject
of $(X\bt Y_1)^* \ot  (X\bt Y_2)$ and so belongs to $\mathcal{S}$.  Therefore, $Y_1^*\ot Y_2\in \tilde\F$,
so  $\deg_\F(Y_1)=\deg_\F(Y_2)$. Similarly, if 
$X_1,\, X_2 \in \O(\C)$ and $Y \in\O(\D)$ are such that that  $X_1\bt Y,\, X_2\bt Y \in  \mathcal{S}$
then $\deg_\E(X_1)=\deg_\E(X_2)$.

Therefore, there is a well-defined isomorphism $f: G_\E \to G_\F$  such that  $f(\deg_\E(X)) =\deg_\F(Y)$
for all $X\in \O(\C)$ and $Y\in \O(\D)$ such that  $X\bt Y \in \O(\mathcal{S})$.  
This means that $\mathcal{S}$ is a fiber product of $\E$ and $\F$. 

It is clear that subcategories $\E,\, \F$ and their gradings are invariants of $\mathcal{S}$.
\end{proof}

\begin{remark}
\label{intersection of subcategories}
Let $(\E_1,\, \F_1,\, G_1)$ and $(\E_2,\, \F_2,\, G_2)$ be subcategory data for $\C\bt \D$. Then
\begin{equation}
\label{intersection of Ss}
\mathcal{S}(\E_1,\, \F_1,\, G_1) \cap \mathcal{S}(\E_2,\, \F_2,\, G_2) 
\cong (\E_1\cap \E_2) \bt_{G_1\times G_2} (\F_1\cap \F_2),
\end{equation}
where the gradings of $\E_1\cap \E_2$  and $\F_1\cap \F_2$ by $G_1\times G_2$ are such that
the $(g_1,\, g_2)$ components are $(\E_1)_{g_1}\cap (\E_2)_{g_2}$ and $(\F_1)_{g_1}\cap (\F_2)_{g_2}$, respectively
(note that these gradings are not faithful in general).
\end{remark}

\subsection{Braided fusion categories and their gradings}
\label{braided section}
 
Let $\B$ be a braided fusion category with a braiding $c_{X,Y}:X\ot Y \to Y \ot X$.  Two objects $X,Y$ of $\B$
{\em centralize} each other if $c_{Y,X}\circ c_{X,Y} =\id_{X\ot Y}$ and {\em projectively centralize} each other 
if $c_{Y,X}\circ c_{X,Y} = \lambda \id_{X\ot Y}$ for some scalar $\lambda \in k$. For a fusion subcategory $\D \subset \B$
its  {\em centralizer} is
\[
\D' =\{Y\in \B \mid  Y \mbox{ centralizes  each } X\in \D \}. 
\]
The symmetric center of $\B$ is $\Z_{sym}(\B):=\B \cap \B'$. We say that $\B$ is {\em non-degenerate} if $\Z_{sym}(\B)=\Vec$. 

For a non-degenerate  $\B$ there is a canonical  non-degenerate bimultiplicative pairing
\begin{equation}
\label{canonical thing}
\langle\, , \, \rangle : \O(\B_{pt}) \times U(\B)\to k^\times
\end{equation}
defined by $c_{Y,X}c_{X,Y}= \langle\, X, \,g \rangle\, \id_{X\ot Y}$ for all $X\in \O(\B_{pt})$ and $Y\in \B_g$, $g\in U(\B)$.
See \cite[3.3.4]{DGNO} for details. 

\begin{proposition}
\label{centralizer of trivial component}
Let $\B$ be a  non-degenerate braided fusion category and let $\D\subset \B$ be a fusion subcategory with a faithful 
grading
\[
\D =\bigoplus_{g\in G}\, \D_g,
\]
where $G$ is an Abelian group.  The centralizer of the trivial component $\D_1$ of $\D$  admits a faithful grading
\[
\D_1' =  \bigoplus_{\phi\in \widehat{G}}\,(\D_1')_\phi,
\]
where $\widehat{G}$ is the group of characters of $G$ and 
\[
(\D_1')_\phi =\{X\in \B \mid c_{Y,X}\circ c_{X,Y} = \phi(g) \id_{X\ot Y},\, \text{for all } Y\in \D_g,\, g\in G\}.
\]
The trivial component of this grading is $\D'$.
\end{proposition}
\begin{proof}
It follows from \cite[3.3]{DGNO} that a simple object $X$ belongs to $\D_1'$ if and only if 
projectively centralizes every simple $Y\in \D$, i.e.,
$c_{Y,X}\circ c_{X,Y} =\lambda_Y  \id_{X\ot Y}$  for some $\lambda_Y \in k^\times$.  
Furthermore, if $Y_1,\, Y_2$ are simple objects lying in $\D_g$ then $\lambda_{Y_1} = \lambda_{Y_2}$.  
Let us denote the latter scalar by $\phi_X(g)$. 
It follows from the braiding axioms that the assignment
\[
\O(\D_1') \to \widehat{G} : X \mapsto \phi_X
\]
is a grading of $\D_1'$ by $\widehat{G}$.  

The fact that the trivial component is $\D'$ and the faithfulness of grading follow from the non-degeneracy of $\B$.
\end{proof}

\subsection{Lagrangian algebras in the center}

For any fusion category $\C$ let $\Z(\C)$ denote its Drinfeld center.

Let $\B$ be a braided fusion category.  A {\em Lagrangian} algebra in $\B$ is a commutative  separable algebra
$A$ in $\B$ such that $\Hom_\B(A,\, \be)\cong k$ and $\FPdim(A)^2 = \FPdim(\B)$. 

Let $I: \C \to \Z(\C)$ denote  the adjoint  of the forgetful functor $F: \Z(\C)\to \C$.   Then $I(\be)$
is a canonical Lagrangian algebra  in $\Z(\C)$.

It was explained in \cite{DMNO} that any braided equivalence
$a: \Z(\C) \xrightarrow{\sim} \B$  gives rise to a Lagrangian  algebra $A= a(I(\be))$ in $\B$.
Conversely, given a Lagrangian algebra $A\in \B$ there is a braided tensor equivalence $\Z(\B_A) \xrightarrow{\sim}  \B$,
where $\B_A$ denotes the fusion category of $A$-modules in $\B$.


\section{Subcategories transversal to a Lagrangian algebra}
\label{sect 3}

\begin{definition}
Let $\C$ be a fusion category, let $\mathcal{B}\subset \C$ be a fusion subcategory, and let
$A$ be an algebra in $\C$.  We will assume that $\Hom_\C(A,\, \be)\cong k$, i.e., that $A$ is a {\em connected} algebra.
We say that $\mathcal{B}$ is {\em transversal} to $A$ if
\begin{equation}
\label{B, A disjoint}
\Hom_\C(X,\, A) = \Hom_\C(X,\, \be)
\end{equation}
for all $X\in \mathcal{B}$.
\end{definition}

In other words, $\mathcal{B}$ is transversal to $A$ if and only if  $\Hom_\C(X,\, A) = 0$ for all non-identity 
$X\in \O(\mathcal{B})$.   

\begin{theorem}
\label{sections as subcategories}
 
Let $\C$ be a fusion category and let $A:=I(\be)$ be the canonical Lagrangian algebra in $\Z(\C)$.
Braidings on $\C$ are in bijection with fusion subcategories $\B \subset \Z(\C)$  transversal to $A$ 
and such that  $\FPdim(\B)=\FPdim(\C)$. 
\end{theorem}
\begin{proof}
It is well known that braidings on a fusion category $\C$  are in bijection with  sections of the forgetful
functor $F:\Z(\C)\to \C$, i.e., with embeddings $\iota:\C \to \Z(\C)$ such that  $F\circ \iota = \id_\C$.
The latter correspond to  fusion subcategories $\B \subset \Z(\C)$
such that the restriction $F|_{\B} :\B \to \C$ is an equivalence. This is equivalent to 
$\FPdim(\B)=\FPdim(\C)$ and $F|_{\B} :\B \to \C$  being injective, i.e., fully faithful.

Note that $F$ is identified with the functor of taking free $A$-modules:
\[
\Z(\C) \to \Z(\C)_A \cong \C :  Z \mapsto A\ot Z. 
\]
Observe that 
\[
\Hom_{\C}(F(Z),\, \be) \cong   \Hom_{\Z(\C)_A}(A\ot Z,\, A)  \cong \Hom_{\Z(\C)}(Z,\, A),
\]
for all $Z\in \Z(\C)$. The injectivity of $F|_{\B}:\B \to \C$ is equivalent to  $\Hom_{\C}(F(Z),\, \be) = \Hom_{\B}(Z,\, \be)$ 
for all $Z\in \B$ and, hence,  to $A$ and $\B$ being transversal.
\end{proof}


\section{Braidings on non-degenerate fusion categories}
\label{sect 4}
\subsection{Classification of braidings}

Let $\B$ be a fusion category with a non-degenerate braiding $c =\{c_{X,Y}\}$.

Any grading of a fusion category $\C$ by a group $G$  
determines a  homomorphism 
\[
h_\C: \O(\C_{pt}) \to G.
\]

\begin{theorem}
\label{braidings on non-degenerate}
The braidings on $\B$ are in bijection with subcategory data  $(\E,\, \F,\, G)$ such that
$\E \vee \F =\B$,  $\E \cap \F$ is pointed, and  $h_\E+h_\F : \O(\F \cap \E) \to G$ is an isomorphism.
Here $\E \vee \F$ denotes the fusion subcategory of  $\B$ generated by $\E$ and $\F$. 
\end{theorem}
\begin{proof}
We will use the characterization of braidings from Theorem~\ref{sections as subcategories}.

Since $\B$ is non-degenerate, we have $\Z(\B) \cong \B \bt \B^\rev$, where $\B^\rev$ denotes $\B$
equipped with the reverse braiding $c^\rev_{X,Y}:=c^{-1}_{Y,X}$.
The forgetful functor $F:\Z(\B)\to \B$ is identified with the
tensor multiplication $\B \bt \B^\rev\to \B$ and the canonical Lagrangian algebra in $\Z(\B)$ is
\[
A=\oplus_{X\in \O(\B)}\, X^*\bt X. 
\]
The notion of a subcategory datum for a tensor product of fusion categories was introduced in Definition~\ref{subcategory datum}.
Suppose that $\mathcal{S}(\E,\, \F,\, G)$ is transversal to $A$ and is such that $\FPdim(\mathcal{S}(\E,\, \F,\, G))=\FPdim(\B)$.
Since the restriction of $F$ on $\mathcal{S}(\E,\, \F,\, G)$ is injective we must have 
\[
\FPdim(F(\mathcal{S}(\E,\, \F,\, G)))=\FPdim(\B).
\]
On the other hand, $\FPdim(F(\mathcal{S}(\E,\, \F,\, G))) \leq  \FPdim(\E\vee \F)$, so $\E\vee \F =\B$.

Using \cite[Lemma 3.38]{DGNO}  we get
\begin{equation}
\label{EFdim}
\FPdim(\mathcal{S}(\E,\, \F,\, G))= \frac{\FPdim(\E)\FPdim(\F)}{|G|} 
= \frac{\FPdim(\E\vee \F)\FPdim(\E\cap \F)}{|G|}.
\end{equation}

It follows from \eqref{EFdim} that  $\FPdim(\E \cap \F) = |G|$.
If $X$ is a non-zero  simple object in $\E_g\cap \F_h$ then $X\ot X^*\in \E_1\cap \F_1$. It follows that $X\ot X^*=\be$ (since other possibilities 
contradict the transversality of $\mathcal{S}(\E,\, \F,\, G$) and $A$).  
Thus, $X$ is invertible and $\E\cap \F$ is pointed. For any non-identity $g\in G$ we must have $\E_g\cap \F_{g^{-1}}=0$.
This is equivalent to the injectivity of  $h_\E+h_\F$. 
Indeed, otherwise there is a nonzero $X\in \E_g$ such that $X^*\in \F_g$ and $X\bt X^*\in \mathcal{S}(\E,\, \F,\, G)$,
contradicting the transversality assumption. 

Since $|\O(\E \cap \F) | =|G|$, $h_\E+h_\F$ is an isomorphism.

Conversely, suppose that a datum $(\E,\, \F,\, G)$ satisfies conditions in the statement of the theorem. 
By \eqref{EFdim},  $\FPdim(\mathcal{S}(\E,\, \F,\, G))=\FPdim(\B)$.  We have  $\E_g\cap \F_{g^{-1}}=0$ for all $g\in G,\, g\neq e$. 
Thus, $\mathcal{S}(\E,\, \F,\, G)$ contains no simple
objects of the form $X^*\bt X$ for $X\neq 1$, i.e., $\mathcal{S}(\E,\, \F,\, G)$ is transversal to $A$.
%
%
\end{proof}

\begin{remark}
\label{explicit braiding on non-degenerate}
Under the conditions of Theorem~\ref{braidings on non-degenerate}, we have $\B \cong \E \bt_G \F$ 
(as a fusion category) and the  corresponding braiding $\tilde{c}$ is given by  
\[
\tilde{c}_{X_1\bt Y_1, X_2\bt Y_2} = c_{X_1,X_2} \bt c^{-1}_{Y_2,Y_1}
\]
for all $X_1\bt Y_1,\, X_2\bt Y_2$ in $\B$. 
\end{remark}

\begin{corollary}
\label{corollary centralizer of SEFG}
Let $(\E,\, \F,\, G)$ be a subcategory datum for $\B \bt \B^\rev$. 
Then 
\begin{equation}
\label{centralizer of SEFG}
\mathcal{S}(\E,\, \F,\, G)' = \bigoplus_{\phi\in \widehat{G}}\, (\E_1')_\phi \bt (\F_1)'_{\phi^{-1}},
\end{equation}
where the $\widehat{G}$-gradings  on $\E_1'$ and $\F_1'$ are defined as in Proposition~\ref{centralizer of trivial component}.
\end{corollary}
\begin{proof}
For all objects $V,W$ let us denote  $\beta_{V,W}:= c_{W,V}\circ c_{V,W}$. 

Let $X\bt Y$ be an object of $\B\bt \B$ and let $X_g \bt Y_g$ be an object of $\mathcal{S}(\E,\, \F,\, G)_g,\, g\in G$. 
Then
\[
\beta_{X\bt Y, X_g \bt Y_g} = \beta_{X, X_g } \bt \beta_{Y,  Y_g}
\]
and so $X\bt Y$ centralizes $X_g \bt Y_g$ if and only if $\beta_{X, X_g }$ and $\beta_{Y,  Y_g}$ are mutually inverse scalars.
This means that $X$ projectively centralizes $\E$ and centralizes $\E_1$ (respectively, $Y$ projectively centralizes $\F$ 
and centralizes $\F_1$).  Thus,
\[
\mathcal{S}(\E,\, \F,\, G)'  =  \E_1' \bt_{\widehat{G}} \F_1' = \bigoplus_{\phi\in \widehat{G}}\, (\E_1')_\phi \bt (\F_1)'_{\phi^{-1}},
\]
as required.
\end{proof}


Let $\B(\F,\, \E,\, G)$ denote the braided fusion category (with underlying fusion category $\B$)
corresponding to the datum $(\E,\, \F,\, G)$  from Theorem~\ref{braidings on non-degenerate}. 

\begin{corollary}
We have $\B(\E,\, \F,\, G)^\rev \cong \B(\E_1',\, \F_1',\, \widehat{G})$, where the fiber product
of $\E_1'$ and $\F_1'$ is as in \eqref{centralizer of SEFG}.
\end{corollary}

\begin{corollary}
\label{sym center pointed}
The symmetric center of $\B(\F,\, \E,\, G)$  has a (not necessarily faithful) grading
\[
\Z_{sym}(\B(\F,\, \E,\, G)) =\bigoplus_{(g,\phi)\in G \times \widehat{G}}\, \Z_{sym}(\B(\F,\, \E,\, G))_{(g,\phi)}, 
\]
where $\Z_{sym}(\B(\F,\, \E,\, G))_{(g,\phi)} \cong (\E_g \cap (\E_1)'_\phi) \bt (\F_g \cap (\F_1)'_{\phi^{-1}})$.
In particular, $\Z_{sym}(\B(\F,\, \E,\, G))$ is pointed.
\end{corollary}
\begin{proof}
The formula for homogeneous components follows from Corollary~\ref{corollary centralizer of SEFG}.
The trivial component of the grading of $\Z_{sym}(\B(\F,\, \E,\, G))$ is contained in $\E_1\bt \F_1$ and so it is 
equivalent to $\Vec$. 
Hence, $\Z_{sym}(\B(\F,\, \E,\, G))$ is pointed.
\end{proof}

\begin{remark}
Corollary~\ref{sym center pointed} means that  if $\B$ has a non-degenerate braiding  then other braidings on $\B$
cannot be ``too symmetric" as the symmetric center remains pointed.  Conversely, if $\B$ has a braiding such that
$\Z_{sym}(\B)$ is not pointed, then no non-degenerate braidings on $\B$ can exist.  In particular, $\Rep(G)$
for a non-abelian $G$ does not admit any non-degenerate braidings (equivalently, there are no modular category structures
on $\Rep(G)$). 
\end{remark}

\begin{proposition}
Let $\B$ be a fusion category that admits a non-degenerate braiding. Then all non-degenerate braidings on $\B$ correspond to 
data $(\E,\, \F,\, G)$ such that 
\begin{equation}
\B \cap \big( (\E_g \cap \F_\phi) \bt (\F_g \cap \E_{\phi^{-1}}) \big) =
\begin{cases}
\Vec & \mbox{ if } g=1, \phi=1 \\
0      & \mbox{ otherwise},
\end{cases}
\end{equation}
where  we use identification  $\B=  \E \bt_G \F \subset \E \bt \F$. 
\end{proposition}
\begin{proof}
Follows  Corollary~\ref{sym center pointed}. 
\end{proof}

\subsection{Braidings on unpointed categories}

Let $\B$ be a fusion category with non-degenerate braiding. Suppose that  $\B_{pt}=\Vec$,
i.e., $\B$ is {\em unpoitned}.
It was shown in \cite{Mu} that in this case there is factorization of $\B$ into a direct product of prime subcategories:
\begin{equation}
\label{prime factorization}
\B = \B_1 \bt \cdots \bt \B_n,
\end{equation}
which is unique up to a permutation of factors.

\begin{corollary}
Let $\B$ be a fusion category such that $\B_{pt}=\Vec$. Suppose that
$\B$ admits a non-degenerate braiding.  Let  \eqref{prime factorization}  be the prime factorization of $\B$.
Then all braidings on $\B$ are non-degenerate and there are
precisely $2^n$ such braidings. The corresponding braided fusion categories are
\begin{equation}
\label{prime factorization rev}
\B = \B_1^\pm \bt \cdots \bt \B_n^\pm,
\end{equation}
where $\B_i^+=\B_i$ and $\B_i^-=\B_i^\rev$ for $i=1,\dots,n$. 
\end{corollary}
\begin{proof}
Since, $\B$ is unpointed, according to Remark~\ref{explicit braiding on non-degenerate} we have $\B \cong \E \bt \F$ as a fusion category.
We claim that $\E$ and $\F$ centralize each other  with respect to the original braiding of $\B$. Indeed, for all $X\in \O(\E)$ and $X\in \O(\F)$
the object $X \bt Y$ is simple and, therefore,
\[
c_{Y,X} \circ c_{X,Y} = \lambda_{X,Y} \id_{X\bt Y},\qquad \lambda_{X,Y} \in k^\times. 
\]
It follows that the map
\[
\O(\E \bt \F) \to k^\times : X \bt Y \mapsto \lambda_{X,Y}
\]
is a grading of $\E \bt \F$ .  But $U(\E \bt \F)\cong \widehat{\O((\E \bt \F)_{pt})}$ is trivial, and so $\lambda_{X,Y}=1$ for all $X,\,Y$,
which proves the claim. It follows that $\E$ and $\F$ must be non-degenerate subcategories  of $\B$. By \cite[Section 2.2]{DMNO}
there is a subset $J\subset \{1,\,\dots,n\}$ such that $\E= \bigoplus_{i\in J}\, \B_i$ and $\F= \bigoplus_{i\not\in J}\, \B_i$. This implies the statement.
\end{proof}

\subsection{Gauging}

Let $\B$ be a non-degenerate braided fusion category with a braiding $c_{X,Y}: X\ot Y \to Y\ot X$. 
A {\em gauging} of $\B$ is  the following procedure of changing the braiding by a bilinear form
$b: U(\B)\times U(\B)\to k^\times$.  A new braiding $\tilde{c}_{X,Y}: X\ot Y \to Y\ot X$ is defined by
\[
\tilde{c}_{X,Y} = b(\deg(X),\, \deg(Y))\,c_{X,Y},
\]
for all $X,Y\in \O(\B)$, where $\deg$ denotes the degree of a simple object with respect to the universal grading.
By definition, gaugings of a given braiding form a torsor over the group of bilinear forms on $U(\B)$.

The corresponding embedding $\B\to \Z(\B)=\B \bt \B^\rev$ is given by $X\mapsto (X\ot V_X )\bt V_X^*$ for all $X\in \O(\B)$, where 
$V_X\in \O(\B_{pt})$ is determined by the condition
\[
\langle V_X, \, y \rangle = b(\deg(X),\, y),\quad  \mbox{for all } y\in U(\B). 
\]
Here $\langle \, , \, \rangle : \O(\B_{pt}) \times \widehat{U(\B)}\to k^\times$ denotes the canonical pairing \eqref{canonical thing}.

In this situation  $\E=\B$, $\F_1 = \Vec$ (so that $\F\subset \B_{pt}$), and $G\subset  \O(\B_{pt})$ is the image of the homomorphism $U(\B)\to \O(\B_{pt}): X \mapsto V_X$. 

Conversely, if a  datum $(\E,\, \F,\, G)$  from Theorem~\ref{braidings on non-degenerate} is such that $\E=\B$ and $\F_1 =\Vec$
(respectively,  $\F=\B$ and $\E_1=\Vec$) then the corresponding braiding is a gauging of the original braiding of $\B$
(respectively,  of the reverse braiding). 

In the next two examples for a finite group $G$ we denote by $\Z(G)$ the center of $\Vec_G$.

\begin{example}
(This result was independently obtained by Costel-Gabriel Bontea using different techniques).
Let $\B:= \Z(S_n),\, n\geq 3$, where $S_n$ denotes the symmetric group on $n$ symbols.  
Observe that $\B$ has a unique maximal fusion subcategory $\B_{ad}$,
which is the subcategory of vector bundles supported on the alternating subgroup $A_n$. Thus, in any presentation
$\B=\E\bt_G \F$ either $\E=\B$ or $\F=\B$. Since $U(\B)=\mathbb{Z}_2$ we must have $G=\{1\}$ or  $G=\mathbb{Z}_2$.
If $\E=\B$ then $\F=\Vec$ or $\F=\B_{pt}$ (note that $\FPdim(\B_{pt})=2$). The first possibility  gives the standard braiding
of $\B$, while the second gives its gauging with respect to the $\mathbb{Z}_2$-grading of $\B$. The situation when $\F=\B$
is completely similar.

Hence, $\B$ has $4$ different braidings: the usual braiding of the center, its reverse, and their gaugings with respect to the 
$\mathbb{Z}_2$-grading of $\Z({S_n})$.  The corresponding data are: $(\B,\, \Vec,\, 1)$, $(\Vec,\,\B,\, 1)$,
$(\B,\, \B_{pt},\, \mathbb{Z}_2)$, and $(\B_{pt},\,\B,\, \mathbb{Z}_2)$, respectively.
\end{example}

\begin{example}
Let $G$ be a non-abelian group of order $8$, i.e., $G$ is either the dihedral group or the quaternion group.
Let $\B=\Z(G)$.  We claim that every braiding of $\B$ is a gauging of either its standard braiding or its reverse.
The structure of $\Z(G)$ was studied in detail by various authors including \cite{GMN, MN}. One has $U(\B)= \mathbb{Z}_2^3$
(so in particular, the standard braiding of $\B$ has $2^9=512$ different gaugings!)
The trivial component of the universal grading is $\B_{pt}=\B_{ad}$, this is a pointed Lagrangian subcategory of the Frobenius-Perron dimension $8$.
Furthermore, for any non-pointed fusion subcategory $\E\subset \B$ its adjoint subcategory $\E_{ad}$ contains at least $4$ invertible objects. 
In any presentation $\B=\E\bt_G \F$ satisfying the conditions of Theorem~\ref{braidings on non-degenerate}
one of the subcategories $\E$, $\F$ must be non-pointed and and another must be pointed. Indeed, if both are pointed then so is $\B$, a contradiction.
If both are non-pointed then $\FPdim(\E_1\cap\F_1)\geq 2$, a contradiction.

Suppose that $\E$ is non-pointed. Then $\F$ is a pointed fusion subcategory of $\B$ with $\FPdim(\B)=1,\,2,\,4 \text{ or } 8$. 

If $\FPdim(\F)=1$ then we get the standard braiding of $\B$.

If $\FPdim(\F)=2$ then either $G=\mathbb{Z}_2$ and the corresponding braiding is a gauging of the standard one, or
$G=\{1\}$ and   $\B =\E\bt \F$. The latter is impossible since in this case $\FPdim(\E)=32$  and $\E$ contains $\B_{pt}$ and, hence, $\F$.

If $\FPdim(\F)=4$  then either $G=\mathbb{Z}_4$ the corresponding braiding is a gauging of the standard one, or
$G=\mathbb{Z}_2$  and so $\FPdim(\E)=32$ and $\FPdim(\E_1\cap\F_1)=2$, a contradiction, or $G=\{1\}$ and   $\B =\E\bt \F$
which is impossible.

Finally, if $\FPdim(\F)=8$ then we must have $G=\mathbb{Z}_2^3$ since otherwise we again have $\FPdim(\E_1\cap\F_1)\geq 2$,
which contradicts conditions of  Theorem~\ref{braidings on non-degenerate}. So in this case $\F_1=\Vec$ and the grading of $\B$
is a gauging of the standard one.

Thus, if $\E$ is non-pointed then the corresponding grading is always a gauging of the standard one. Switching $\E$ and $\F$ will give
gaugings of the reverse brading. 
\end{example}


\section{Braidings on group-theoretical categories}
\label{sect 5}

Let $G$ be a finite group. Let us denote $\C(G)=\Vec_G$ and $\Z(G):=\Z(\Vec_G)=\C(G)^G$.

\subsection{Lagrangian algebras in the center of $\Vec_G$}

It is well known that  $\Z(G)$ is identified with the category of $G$-equivariant vector bundles on $G$. 
The isomorphism classes of simple objects of $\Z(G)$ are parameterized by pairs $(K,\, \pi)$, where $K\subset G $ is a conjugacy class
and $\pi$ is the isomorphism class an irreducible representation of the centralizer $C_G(g_K)$ of  $g_K\in K$.  The corresponding 
object $V(K,\, \pi) =\oplus_{g\in K}\, V(K,\, \pi)_g$ is the vector bundle supported on $K$ whose  equivariant structure restricted to $C_G(g_K)$
acts by $\pi$ on $V(K,\, \pi)_{g_K}$. 

Recall that  equivalence classes of indecomposable  $\C(G)$-module categories are parameterized by 
conjugacy classes of pairs
$(H,\, \mu)$, where $H$ is a subgroup of $G$ and $\mu\in H^2(G,\, k^\times)$. 
The module category $\M(H,\, \mu)$ corresponding to $(H,\, \mu)$ is the category of modules
over the twisted group algebra $k_\mu[H]$ in $\C(G)$.  It can be identified with a certain
category of  $H$-invariant vector bundles on $G$. 

 Let $\Z(G;\, H)= \Vec_G^H$ be the category of $H$-equivariant objects in $\Vec_G$. We have $\Z(H)\subset \Z(G;\, H)$.  There is an obvious
forgetful functor $F_H: \Z(G)\to \Z(G;\, H)$. Let $I_H: \Z(G;\, H) \to \Z(G)$ denote its adjoint.  

The following construction was given in \cite{D2}.
The twisted group algebra $k_\mu[H]$ is a Lagrangian algebra in $\Z(H)$ with the obvious grading
and the $H$-equivariant  structure given by
\begin{equation}
\label{eqstronkmu}
k_\mu[H] \to g k_\mu[H] g^{-1} :  x \mapsto \eps_g(x) gxg^{-1}, \quad \mbox{where\,  }
\eps_g(x) =\frac{\mu(gxg^{-1},\, g)}{\mu(g,\,x)},\qquad g,x\in H.
\end{equation}
Here we abuse notation and identify the cohomology class $\mu$ with a $2$-cocycle representing it. 
Note that 
\begin{equation}
\label{d2eps}
\frac{\eps_g(x)\eps_g(y)}{\eps_g(xy)} =\frac{\mu(gxg^{-1},\,gyg^{-1})}{\mu(x,\,y)},\qquad g,\,x,\,y\in H. 
\end{equation}
In particular, $\eps_{g_K}$ restricts to a linear character of $C_G(g_K)$. 
As an object of $\Z(H)$, 
\begin{equation}
\label{kmuH as object}
k_\mu[H] \cong \bigoplus_K\, V(K,\, \eps_{g_K}).
\end{equation}
Let $A(H,\, \mu) \in \Z(G)$ be the Lagrangian algebra  corresponding to  the $\C(G)$-module category  $\M(H,\, \mu)$.
It was shown in \cite[Section 3.4]{D2} that 
\begin{equation}
\label{AH-mu = I-muH}
A(H,\, \mu)  \cong I_H(k_\mu[H]).
\end{equation}
Here $k_\mu[H]\in \Z(H)$ is considered as an algebra in $\Z(G;\, H)$. 

\subsection{Transversality criterion and parameterization of braidings}

Tensor subcategories of $\Z(G)$ were classified in \cite{NNW}. They are in bijection with 
triples $(L,\,M,\, B)$, where 
\begin{enumerate}
\item[(T1)] $L$ and $M$ are normal subgroups of $G$ commuting with each other, 
\item[(T2)] $B:L\times M \to k^\times$ is a $G$-invariant bicharacter.  
\end{enumerate}
The corresponding subcategory
$\mathcal{S}_G(L,\,M,\, B)$ consists of vector bundles supported on $L$ and such that the restriction of their
$G$-equivariant structure  on $M$ is the scalar multiplication by 
$B(g,\, -)$ for all $g\in L$.  Equivalently, simple objects of $\mathcal{S}_G(L,\,M,\, B)$ are objects $V(K,\, \pi)$,
where $K$ is a conjugacy class contained in $L$ and $\pi$ is contained in
the induced representation  $\Ind_{M}^{C(g_K)}\, B(g_K,\,-)$. 
We have
\begin{equation}
\label{FP dim SG}
\FPdim(\mathcal{S}_G(L,\,M,\, B)) = |L| [G:M].
\end{equation}
We denote by  $\widehat{B}: L \to \widehat{M}$ the group homomorphism associated to $B$. 

Let $\mu$ be a $2$-cocycle on $G$ with values in $k^\times$.  The map $\Alt(\mu): C_G(M) \times M \to k^\times$ 
defined by 
\begin{equation}
\label{epsmu}
\Alt(\mu) (g,\, x) =\frac{\mu(x,\,g)}{\mu(g,\, x)},\qquad g\in C_G(M),\,x\in M.
\end{equation}
is bimultiplicative and $G$-invariant. We have 
\[
\Alt(\mu) (g,\, x)=\eps_g(x)
\]
for all  $g\in C_G(M),\,x\in M$.

\begin{lemma}
\label{SKLB transversal to k-muH}
The subcategory $\mathcal{S}_G(L,M,B) \subset \Z(G)$ is transversal to the Lagrangian
algebra $k_\mu[G]$ if and only if $\frac{B} {\Alt(\mu)}: L \times M \to k^\times$ is non-degenerate
in the second argument, i.e.,  for all $g\in L,\, g\neq 1,$ there is $x\in M$ such that   
\begin{equation}
\label{inequality}
\frac{B} {\Alt(\mu)} \left(g,\, x \right) \neq 1.
\end{equation}
\end{lemma}
\begin{proof}
By \eqref{kmuH as object}  the transversality is equivalent to the condition
\[
\Hom_{C_G(g_K)}(\eps_{g_K},\, \Ind_M^{C_G(g_K)} \widehat{B}(g_K)) =0
\]
for all non-identity conjugacy classes $K\subset L$.
By the Frobenius reciprocity this is equivalent to 
\[
\Hom_{M}(\eps_{g_K}|_M,\,  \widehat{B}(g_K)) =0, \qquad K\subset L,\, K\neq \{1\},
\]
i.e., $\eps_{g_K}|_M \neq \widehat{B}(g_K)$ for all non-identity $K$.
This condition means that for each $g_K$ with  $K\subset L$ ($K~\neq~\{1\}$) there is $x\in M$ such that
\[
\frac{\mu(x,\, g_K)}{\mu(g_K,\, x)} \neq B(g_K,\,x).
\]
Using the $G$-invariance of $B$ and $\Alt(\mu)$ we get the result. 
\end{proof}


\begin{theorem}
\label{GT braidings}
Braidings on $\C(G)_{\M(H,\,\mu)}^*$  are in bijection with triples $(L,\, M, B)$ 
satisfying  (T1),  (T2) and the following conditions:
\begin{enumerate}
\item[(i)]  $LH = MH = G$,  
\item[(ii)] the restriction of  $\frac{B} {\Alt(\mu)}$  on  $(L\cap H)\times (M\cap H)$ is non-degenerate. 
\end{enumerate}
\end{theorem}
\begin{proof}
By Theorem~\ref{sections as subcategories} braidings on $\C(G)_{\M(H,\,\mu)}^*$ 
are parameterized by fusion subcategories $\mathcal{S}_G(L,M,B) \subset \Z(G)$ of the Frobenius-Perron 
dimension $|G|$ transversal to $A(H,\, \mu)$. 

The above dimension condition is equivalent to $|L|=|M|$ by \eqref{FP dim SG}. 
Note that this condition follows from (i) and (ii).


In view of \eqref{AH-mu = I-muH} we see that a necessary condition for the above transversality
is that the restriction of the forgetful functor $F_H: \Z(G)\to \Z(G;\,H)$ to $\mathcal{S}_G(L,M,B)$ is injective.
The latter condition is equivalent to transversality of $\Rep(G/M)$ and the function algebra $\Fun(G/H,\, k)$
in $\Rep(G)$, in other words, to $\Hom_G(\Ind_M^G k^\times,\, \Ind_H^G k^\times) =k$. Here $k^\times$ denotes
the trivial module.  By the Mackey restriction formula the latter is equivalent to $MH=G$. 

If the above condition is satisfied, the transversality of
$\mathcal{S}_G(L,M,B)$ and $A(H,\, \mu)$ in $\Z(G)$ is equivalent to the transversality
of $F_H(\mathcal{S}_G(L,M,B))$ and $k_\mu[H]$ in $\Z(G;\,H)$.

In this case we have
\[
F_H(\mathcal{S}_G(L,M,B)) \cap \Z(H) =\mathcal{S}_H(L\cap H,\,M\cap H,\, B|_{(L\cap H)\times (M \cap H)}).
\]
Now we can apply  Lemma~\ref{SKLB transversal to k-muH} (with $G$ replaced by $H$). 
The transversality of the subcategory $\mathcal{S}_H(L\cap H,\,M\cap H,\, B|_{(L\cap H)\times (M \cap H)})$ and 
the algebra $k_\mu[H]$ is equivalent to the injectivity of the corresponding homomorphism  $L\cap H \to \widehat{M \cap H}$, 
whence $|L\cap H| \leq |M\cap H|$.  This implies
\[
|LH| = \frac{|L||H|}{|L\cap H|} \geq   \frac{|M||H|}{|M\cap H|} =|MH|=|G|,
\]
so that $LH=G$,  $|L\cap H|=|M\cap H|$ and, hence, 
$\frac{B}{\Alt(\mu)}|_{(L\cap H)\times (M \cap H)}$
is non-degenerate.
\end{proof}

\begin{example}
Let $G$ be a non-abelian group and let $H \subset  G$ be a subgroup such that the only normal subgroup
$N$ of $G$ such that $HN=G$  is  $G$ itself. Then $\C(G)^*_{\M(H,\, \mu)}$ does not admit a braiding.
In particular, if $G$ is simple non-abelian and $H\neq G$ then $\C(G)^*_{\M(H,\, \mu)}$ does not admit a  braiding.
\end{example}

\begin{remark}
\begin{enumerate}
\item[(i)] Masuoka \cite{Ma} showed that that certain self-dual non-commutative and non-cocommutative 
semisimple Hopf algebras of dimension $p^3$, where $p$ is an odd prime, admit no quasi-triangular structures, 
thus giving examples of group-theoretical fusion categories that do not admit any braiding. 
These examples, however, are not of the form considered in Theorem~\ref{GT braidings} (to obtain them
one has to generalize Theorem~\ref{GT braidings}  by replacing  $\Vec_G$ with $\Vec_G^\omega$ for a 
non-trivial $3$-cocycle $\omega$).
\item[(ii)] An absence of braidings on certain group-theoretical categories associated to exact factorizations of almost simple
groups was established by Natale \cite{N}.
\end{enumerate}
\end{remark}

\begin{example}
The category $\C(G)^*_{\M(G,\, 1)}$ is equivalent to $\Rep(G)$. In this case Theorem~\ref{GT braidings} says
that braidings on $\Rep(G)$ are in bijections of triples $(L,\,M,\, B)$, where $L$ and $M$ are normal  subgroups
of $G$ commuting with each other and $B:L\times M \to k^\times$ is a non-degenerate $G$-invariant bilinear form
(note that these conditions on $B$ imply that $L$ and $M$ must be Abelian).
This classification was obtained by Davydov  \cite{D1}.
\end{example}

\begin{example}
\label{Keilberg}
The category $\Z(G)$ is equivalent to $\C(G\times G^{\op} )^*_{\M(D,\, 1)}$, where $G^{\op}$ is the group with the opposite
multiplication and $D=\{ (g,\,g^{-1}) \mid g\in G\}$.  Braidings on this category are parameterized by triples $(L,\, M,\, B)$,
where $L,\, M$ are  normal subgroups of $G\times G^{\op}$ and $B: L\times M \to k^\times$ is a  $G\times G^{\op}$-invariant bilinear form 
such that the following conditions  are satisfied:
\begin{enumerate}
\item[(i)] $L$ and $M$ commute with each other, 
\item[(ii)] $LD=MD = G\times G^{op}$, and
\item[(iii)] the restriction of $B$ on $(L\cap D)\times(M \cap D)$ is non-degenerate. 
\end{enumerate}
The standard braiding of $\Z(G)$ corresponds to $L = G \times 1, \, M=1\times G^{op}$, and trivial $B$. 

This parameterization is an alternative to the description of quasitriangular structures on the Drinfeld double
of $G$ given by Keilberg \cite{K}.
\end{example}

\bibliographystyle{ams-alpha}

\end{document}